\documentclass[11pt]{article}
\usepackage{epsfig,enumerate,amsmath,amsfonts,amsbsy,amssymb,amsthm,mathrsfs,lineno,ifpdf}
\usepackage[mathscr]{euscript}
\usepackage[numbers]{natbib}
\usepackage{hyperref, url}
\usepackage{setspace,graphicx}
\usepackage[usenames,dvipsnames]{pstricks}
\usepackage[lined,boxruled,linesnumbered,procnumbered]{algorithm2e}
\usepackage[margin=1.2in]{geometry}
\newtheorem{theorem}{Theorem}
\newtheorem{lemma}{Lemma}
\newtheorem{quest}{Question}

\newtheorem{coro}{Corollary}
\newtheorem{obs}{Observation}

\newtheorem{claim}{Claim}

\usepackage[normalem]{ulem}
\usepackage{enumitem}
\usepackage{etoolbox}

\let\oldenumerate\enumerate
\renewcommand{\enumerate}{
	\oldenumerate
	\setlength{\itemsep}{1.5pt}
	\setlength{\parskip}{0pt}
	\setlength{\parsep}{0pt}
}

\numberwithin{equation}{section}

\hypersetup{
	colorlinks=true,
	linkcolor=blue,
		filecolor=magenta,
	pdfpagemode=FullScreen,
}
\newcommand{\epn}{{\rm epn}}

\newcommand{\paw}{{\rm paw}}
\newcommand{\sdsp}{{\rm SDS-$P_5$-FREE}}
\begin{document}

\title{Secure domination in $P_5$-free graphs}

\author{$^{1}$Uttam K. Gupta, $^2$Michael A. Henning\thanks{Research supported in part by the University of Johannesburg and the South African National Research Foundation}, $^{3}$Paras Vinubhai Maniya\thanks{Corresponding author.}, and $^3$Dinabandhu Pradhan  \\ \\
$^{1}$Anugrah Memorial College Gaya\\
	Magadh University, Bodh Gaya, Bihar, India \\
	\small \tt Email: Email: ukumargpt@gmail.com \\ \\
$^{2}$Department of Mathematics and Applied Mathematics \\
	University of Johannesburg \\
	Auckland Park, 2006 South Africa\\
	\small \tt Email: mahenning@uj.ac.za \\ \\
$^{3}$Department of Mathematics \& Computing\\
Indian Institute of Technology (ISM) \\
Dhanbad, India \\
\small \tt Email: maniyaparas9999@gmail.com \\
\small \tt Email: dina@iitism.ac.in  \\ }

\date{}
\maketitle

\begin{abstract}
A dominating set of a graph $G$ is a set $S \subseteq V(G)$ such that every vertex in $V(G) \setminus S$ has a neighbor in $S$, where two vertices are neighbors if they are adjacent. A secure dominating set of $G$ is a dominating set $S$ of $G$ with the additional property that for every vertex $v \in V(G) \setminus S$, there exists a neighbor $u$ of $v$ in $S$ such that $(S \setminus \{u\}) \cup \{v\}$ is a dominating set of $G$. The secure domination number of $G$, denoted by $\gamma_s(G)$, is the minimum cardinality of a secure dominating set of $G$.  We prove that if $G$ is a $P_5$-free graph, then $\gamma_s(G) \le \frac{3}{2}\alpha(G)$, where $\alpha(G)$ denotes the independence number of $G$. We further show that if $G$ is a connected $(P_5, H)$-free graph for some $H \in \{ P_3 \cup P_1, K_2 \cup 2K_1, \paw, C_4\}$, then $\gamma_s(G)\le \max\{3,\alpha(G)\}$. We also show  that if $G$ is a $(P_3 \cup P_2)$-free graph, then $\gamma_s(G)\le \alpha(G)+1$.
\end{abstract}



\section{Introduction}

All the graphs considered in this paper are finite, simple, and undirected. We use $P_n$, $C_n$, and $K_n$ to denote the \emph{path}, the \emph{cycle}, and the \emph{complete graph}, respectively on $n$ vertices. For a graph $G$, we use $V(G)$ and $E(G)$ to denote the vertex set and the edge set of $G$, respectively. Two vertices $u$ and $v$ of $G$ are \emph{adjacent} if $uv\in E(G)$. The $\emph{neighbors}$ of $v$ in $G$ are the vertices adjacent to $v$ in $G$. For vertex disjoint graphs $G$ and $H$, we use $G\cup H$ to denote the disjoint union of $G$ and $H$. For a positive integer $m$, we use $mG$ to denote the disjoint union of $m$ copies of $G$. Let $\mathcal{F}$ be a family of graphs. A graph $G$ is called \emph{$\mathcal{F}$-free} if $G$ does not contain any graph in $\mathcal{F}$ as an induced subgraph. When $\mathcal{F}=\{H\}$ (resp. $\{H_1,\ldots,H_k\}$, $k\ge 2$), we use $H$-free (resp. $(H_1,\ldots,H_k)$-free) graphs to denote the $\mathcal{F}$-free graphs. For some special graphs mentioned in this paper, we refer to Figure~\ref{special}.

\begin{figure}[htb]
		\begin{center}
			\includegraphics[scale=0.83]{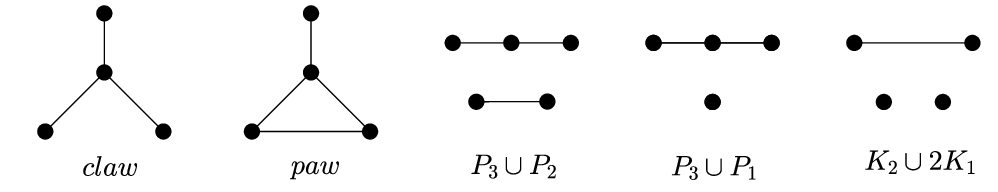}
		\caption{Some special graphs}\label{special}
		\end{center}
\end{figure}

A set $D\subseteq V(G)$ is a \emph{dominating set} of $G$ if every vertex in $V(G) \setminus D$ is adjacent to at least one vertex in $D$. The \emph{domination number} of $G$, denoted by $\gamma(G)$, is the minimum cardinality of a dominating set of $G$. The domination problem and its variants have been extensively studied in the literature. Domination in graphs and its variants have numerous applications and have been thoroughly explored from  theoretical and algorithmic perspectives. For a detailed study of domination and its variants, we refer the reader to~\cite{Haynes1,Haynes2,Haynes3,HeYe-book}.

Secure domination in graphs is a variant of domination introduced by Cockayne et al.~\cite{Cockayne05} in 2005. Since its birth, secure domination remained an active area of research (see~\cite{Burger08,  Burger16, Cockayne03, Klostermeyer08}). Consider a problem of defending the vertices of a graph $G$ with guards against an attacker. The vertices of $G$ may represent regions and at each vertex, at most one guard is located. The attacker can attack at one vertex of $G$ at a time and can do two consecutive attacks. A guard can protect its position vertex and its every neighbor (possibly by moving to it). To defend against the attacker, only one guard can move at a time. If the attacker has already attacked one of the vertices of $G$ where a guard is not placed, then in that case, one guard can move from its current position vertex to an adjacent vertex to stop the attack. Additionally, guards can still protect the vertices of $G$ if a second attack occurs. The idea of secure domination was created as a response to this need, where the arrangements of defending the vertices of $G$ both before and after the attack have been prevented.
	
	In order to fulfill the above requirement, a secure dominating set is formally defined as follows. A set $S \subseteq V(G)$ is called a \emph{secure dominating set} of $G$ if $S$ is a dominating set of $G$ and for each $v\in V(G)\setminus S$, there exists a neighbor $u$ of $v$ in $S$ such that $(S\setminus \{u\})\cup \{v\}$ is a dominating set of $G$. The \emph{secure domination number} of $G$, denoted by $\gamma_s(G)$, is the minimum cardinality of a secure dominating set of $G$. A set $I\subseteq V(G)$ is an \emph{independent set} if no two vertices in the set $I$ are adjacent to each other. The \emph{independence number} of $G$, denoted by $\alpha(G)$, is the maximum cardinality of an independent set of $G$. An independent set in $G$ of cardinality $\alpha(G)$ we call an $\alpha$-set of $G$. If $I$ is a maximal independent set of $G$ and $v$ is an arbitrary vertex in $V(G)\setminus I$, then by the maximality of $I$ the set $I \cup \{v\}$ is not an independent set, implying that the vertex $v$ has at least one neighbor in $I$. We state this well-known observation formally as follows.
	
\begin{obs}
\label{dom}
Every maximal independent set of a graph $G$ is a dominating set of $G$.
\end{obs}

\subsection{Related works}	

By Observation~\ref{dom}, for every graph $G$ we have $\gamma(G)\le \alpha(G)$. We note that $\gamma_s(C_5)=3$ and $\alpha(C_5)=2$, implying that $\gamma_s(G)$ can be larger than $\alpha(G)$. Merouane and Chellali~\cite{merouane15} established the following upper bound on the secure domination number of a graph in terms of its independence number.

\begin{theorem}[\cite{merouane15}]\label{secure-versus-indep}
If $G$ is a graph, then $\gamma_s(G) \le 2\alpha(G) - 1$.
\end{theorem}

A natural question arises whether we can improve the upper bound in Theorem~\ref{secure-versus-indep} to $\gamma_s(G) \le c \alpha(G)$, where $c < 2$ for specific classes of graphs. Cockayne et al.~\cite{Cockayne03} showed that $c = \frac{3}{2}$ for the class of claw-free graphs.

\begin{theorem}[\cite{Cockayne03}]\label{claw-free}
If $G$ is a claw-free graph, then $\gamma_s(G)\le \frac{3}{2}\alpha(G)$.
\end{theorem}

Since the class of $3K_1$-free graphs is a subclass of the class of claw-free graphs, the upper bound in Theorem~\ref{claw-free} also  holds in the class of $3K_1$-free graphs. We note that if $G$ is a $3K_1$-free graph, then $\alpha(G) \le 2$. Hence if $G$ is a $3K_1$-free graph, then $\gamma_s(G)\le 3$. Merouane and Chellali~\cite{merouane15} proved that the upper bound in Theorem~\ref{claw-free} also holds in the class of $C_3$-free graphs, that is, in $K_3$-free graphs.

\begin{theorem}[\cite{merouane15}]\label{$C_3$-free}
If $G$ is a $C_3$-free graph, then $\gamma_s(G) \le \frac{3}{2}\alpha(G)$.
\end{theorem}

We note that the upper bound in Theorem~\ref{$C_3$-free} is tight since if $G$ is the graph obtained from the disjoint union of copies of $C_5$, then the resulting graph $G$ is $C_3$-free and satisfies $\gamma_s(G)=\frac{3}{2}\alpha(G)$. Merouane and Chellai~\cite{merouane15} also showed that if $G$ is a bipartite graph, then $\gamma_s(G)\le \alpha(G)$. Recently, Degawa and Saito~\cite{degawa23} established that the secure domination number of a $C_5$-free graph is at most its independence number.

\begin{theorem}[\cite{degawa23}]\label{cycle}
If $G$ is a $C_5$-free graph, then $\gamma_s(G)\le \alpha(G)$.
\end{theorem}

We note that the bounds for bipartite graphs and $C_5$-free graphs are tight since a star graph $G = K_{1,n}$ is a $C_5$-free bipartite graph and satisfies $\gamma_s(G)=\alpha(G)$~\cite{Cockayne05}.

\section{Main results and motivation}

The class of $P_4$-free graphs is a subclass of the class of $C_5$-free graphs. Hence as a consequence of Theorem~\ref{cycle}, if $G$ is a $P_4$-free graph, then $\gamma_s(G) \le \alpha(G)$. Moreover, two linear algorithms have been independently developed to compute the secure domination number of a $P_4$-free graph (see~\cite{Araki19, jha19}). Hence it is of interest to find a superclass of the class of $P_4$-free graphs that yields a tight upper bound on the secure domination number of a graph in terms of its independence number.
In this paper, we consider such a superclass of $P_4$-free graphs, namely $P_5$-free graphs. We shall prove the following result, a proof of which is given in Section~\ref{Sect:P5-free}.

\begin{theorem}
\label{thm:main1}
If $G$ is a $P_5$-free graph, then $\gamma_s(G) \le \frac{3}{2}\alpha(G)$.
\end{theorem}

Given a graph $G$ and a positive integer $k$, the secure domination problem is to decide whether $G$ has a secure dominating set of cardinality at most~$k$. Merouane and Chellali~\cite{merouane15} showed that secure domination problem is $\mathbf{NP}$-complete for split graphs. Foldes and Hammer~\cite{Foldes} proved that $G$ is a split graph if and only if $G$ is a $(2K_2,C_4,C_5)$-free graph. Since the class of split graphs is a subclass of the class of $C_5$-free graphs, as a consequence of Theorem~\ref{cycle} we infer that if $G$ is a split graph, then $\gamma_s(G)\le \alpha(G)$. Recently, Chen et al.~\cite{Chen25} have proved that if $G$ is a $2K_2$-free graph, then $\gamma_s(G)\le \alpha(G)+1$. Note that the class $2K_2$-free graphs is a superclass of the class of split graphs. We are therefore interested in a superclass of the class of $2K_2$-free graphs that yields a tight upper bound on the secure domination number in terms of its independence number.  We consider such a superclass of $2K_2$-free graphs, namely $(P_3\cup P_2)$-free graphs. We shall prove the following result, a proof of which is given in Section~\ref{Sect:P3P2-free}.

\begin{theorem}
\label{thm:main2}
If $G$ is a $(P_3\cup P_2)$-free graph, then $\gamma_s(G) \le \alpha(G) + 1$.
\end{theorem}

 Although the bound of Theorem~\ref{thm:main2} is tight, we show that the bound can be improved for some subclasses of $(P_3\cup P_2)$-free graphs. A proof of the following result is given in Section~\ref{Sect:P3P2-free}.

\begin{theorem}
\label{thm:main3}
If $G$ is an $H$-free graph, where $H \in \{P_3\cup P_1, K_2\cup 2K_1\}$, then $\gamma_s(G) \le \max\{3,\alpha(G)\}$.
\end{theorem}

We also investigate a subclass of the $P_5$-free graphs, namely connected $(P_5,H)$-free graphs where $H \in \{\paw, C_4\}$. By using the structural properties of such graphs, we shall prove the following results, proofs of which are given in Sections~\ref{Sect:P5C3-free} and~\ref{Sect:P5C4-free}, respectively.

\begin{theorem}
\label{thm:main4}
If $G$ is a connected $(P_5,\paw)$-free graph, then $\gamma_s(G) \le \max\{3,\alpha(G)\}$.
\end{theorem}

\begin{theorem}
\label{thm:main5}
If $G$ is a connected $(P_5,C_4)$-free graph, then $\gamma_s(G) \le \max\{3,\alpha(G)\}$.
\end{theorem}

We summarize our results in Table~\ref{bound-table}. In this table, we present upper bounds on the secure domination number, $\gamma_s(G)$, of a graph $G$ depending on the structural properties of the graph $G$.

\medskip
\begin{table}[htb]
		\begin{center}
			\setlength{\arrayrulewidth}{0.2mm}
			\begin{tabular}{ |l|c|c|}
				\hline
				{\bf Graph Class $\mathcal{C}$} & {\bf $\gamma_s$-bound for $G\in \mathcal{C}$} & {\bf reference} \\
				\hline
				General graphs & $2\alpha(G)-1$&  \cite{merouane15} \\
				\hline
				Claw-free graphs & $\frac{3}{2}\alpha(G)$&  \cite{Cockayne03} \\
				\hline
				$C_3$-free graphs & $\frac{3}{2}\alpha(G)$&  \cite{merouane15} \\
				\hline
				Paw-free graphs & $\frac{3}{2}\alpha(G)$ & [Corollary \ref{Paw}] \\
				\hline
				Bipartite graphs & $\alpha(G)$ &  \cite{merouane15} \\
				\hline
				$C_5$-free graphs & $\alpha(G)$&  \cite{degawa23} \\
				\hline
				$P_5$-free graphs & $\frac{3}{2}\alpha(G)$ & [Theorem~\ref{thm:main1}] \\
				\hline
				$(P_3\cup P_2)$-free graphs & $\alpha(G)+1$ & [Theorem~\ref{thm:main2}] \\
				\hline
				$(P_3\cup P_1)$-free graphs & $\max\{3,\alpha(G)\}$ & [Theorem \ref{P3UP1}] \\
				\hline
				$(K_2\cup 2K_1)$-free graphs &  $\max\{3,\alpha(G)\}$ & [Theorem \ref{K2U2K1}] \\
				\hline
				Connected $(P_5,\paw)$-free graphs & $\max\{3,\alpha(G)\}$ &  [Theorem~\ref{thm:main4}] \\
				\hline
				Connected $(P_5,C_4)$-free graphs & $\max\{3,\alpha(G)\}$ &  [Theorem~\ref{thm:main5}] \\
				\hline
			\end{tabular}\\
			\label{bound-table}
			\vspace{0.2cm}
			\caption{Upper bounds on the secure domination number of a graph $G$}
		\end{center}
\end{table}

\section{Notation, terminology and preliminary results}
\label{prelim}
	
Let $G$ be a graph. The \emph{(open) neighborhood} of a vertex $v\in V(G)$, denoted by $N_G(v)$, is the set of neighbors of $v$ in $G$. The \emph{closed neighborhood} of a vertex $v\in V(G)$ is the set $N_G[v]=N_G(v)\cup \{v\}$. For a set $X\subseteq V(G)$, the \emph{(open) neighborhood} of $X$, denoted by $N_G(X)$, is the set $\{v\in V(G)\setminus X \, \colon v \text{ is adjacent to a vertex in } X\}$. The \emph{closed neighborhood} of $X$ is the set $N_G[X]= N_G(X)\cup X$. We simply use $N(v), N[v], N(X)$, and $N[X]$ to denote $N_G(v), N_G[v], N_G(X),$ and  $N_G[X]$, respectively when the context of the graph $G$ is clear. We use $G[X]$ to denote the subgraph of $G$ induced by the vertices in $X$. A vertex $v$ dominates itself and all its neighbors. A set $X\subseteq V(G)$ is called a \emph{clique} of $G$ if every pair of vertices of $X$ are adjacent. For disjoint subsets $X$ and $Y$ of $V(G)$, $G[X,Y]$ denotes the set of edges $\{e\in E(G)  \, \colon e$ has one endpoint in $X$ and the other endpoint in $Y\}$. If the graph $G$ is clear from the context, we write $[X,Y]$ rather than $G[X,Y]$. We say that $G[X,Y]$ is complete if every vertex in $X$ is adjacent to every vertex in $Y$ in $G$. For a given positive integer $k$, we use the notation $[k]$ to denote the set $\{1,\ldots, k\}$.

For a set $D \subseteq V(G)$ and a vertex $u \in D$, the \emph{$D$-external private neighborhood} of $u$, denoted by $\epn(u,D)$, is the set $\{v \in V(G) \setminus D  \, \colon  N(v) \cap D = \{u\}\}$. A vertex in $\epn(u,D)$ is a \emph{$D$-external private neighbor} of $u$. We say that a vertex $v\in V(G)\setminus D$ is $D$-\emph{defended} if there exists a vertex $u \in N(v)\cap D$ such that $(D\setminus \{u\})\cup \{v\}$ is a dominating set of $G$. Specifically, we say that $v\in V(G)\setminus D$ is $D$-\emph{defended by a vertex} $u$ if $u\in N(v)\cap D$ and $(D\setminus \{u\})\cup \{v\}$ is a dominating set of $G$. We note that $D$ is a secure dominating set of $G$ if and only if every vertex in $V(G) \setminus D$ is $D$-defended. We also note that if $D$ is a dominating set of $G$ and $\epn(u,D)=\emptyset$ for some $u\in D$, then every neighbor of $u$ in $V(G)\setminus D$ is $D$-defended by $u$. We state this observation formally as follows.
	
\begin{obs}
\label{epnempty}
If $D$ is a dominating set of a graph $G$ and $\epn(u,D)=\emptyset$ for some vertex $u \in D$, then every neighbor of $u$ that belongs to $V(G)\setminus D$ is $D$-defended by $u$.
\end{obs}

As a consequence of Observation~\ref{epnempty}, if $D$ is a dominating set of $G$ and $\epn(u,D)=\emptyset$ for every $u\in D$, then every vertex in $V(G) \setminus D$ is $D$-defended, implying that $D$ is a secure dominating set of $G$. For a given dominating set $D$ of $G$, we partition $D$ as $A_D\cup B_D$, where the sets $A_D$ and $B_D$ are defined as follows.
\begin{align*}
\tag{$\mathcal{P}_1$}
\label{indeppart}
    \bullet	\, &  \, A_D= \{u\in D \, \colon  \text{ there exists a neighbor of } u \text{ in } V(G)\setminus D \text{ that is not } D\text{-defended.}\} \\
    \bullet	\, &  \,  B_D=D\setminus A_D.
\end{align*}

We now prove the following useful lemma. We remark that some of the facts that appear in the following lemma have appeared implicitly in several papers (see~\cite{Cockayne03,merouane15}).

\begin{lemma}
\label{mainthmlem-1}	
Let $D$ be a dominating set of $G$ and $A_D\cup B_D$ be a partition of $D$ as defined in~{\rm (}\ref{indeppart}{\rm )}. \\[-24pt]
\begin{enumerate}
\item If $A_D=\emptyset$, then $D$ is a secure dominating set of $G$.
\item If a vertex $v$ in $V(G)\setminus D$ is not $D$-defended, then $N(v)\cap D\subseteq A_D$.
\item If $u$ is a vertex in $A_D$, then $\epn(u,D)\ne  \emptyset$.
\item If a vertex $v$ in $V(G)\setminus D$ is not $D$-defended, then for every neighbor $u$ of $v$ in $D$, the vertex $v$ has a non-neighbor in $\epn(u,D)$, that is, $\epn(u,D) \setminus N(v) \ne \emptyset$.
\item If $D'$ is a superset of $D$, then $A_{D'}\subseteq A_D$ and $B_D\subseteq B_{D'}$, where $A_{D'}\cup B_{D'}$ is a partition of $D'$ as defined in~{\rm (}\ref{indeppart}{\rm )}.
\end{enumerate}
\end{lemma}
\begin{proof}
Parts~(a) and (b) follow by the definition of the sets $A_D$ and $B_D$, while part~(c) follows by Observation~\ref{epnempty}.

To prove part~(d), let $v$ be a vertex in $V(G)\setminus D$ that is not $D$-defended. Suppose, to the contrary, that there exists a neighbor $u \in D$ of $v$ such that $v$ is adjacent to every vertex in $\epn(u,D)$. We note that this also includes the case when $\epn(u,D)=\emptyset$. However, then, $(D\setminus\{u\}) \cup \{v\}$ is a dominating set of $G$, contradicting the fact that the vertex $v$ is not $D$-defended.

To prove part~(e), let $D'$ be a superset of $D$. Since the property of being a dominating set is superhereditary, $D'$ is a dominating set of $G$. We partition $D'$ as $A_{D'}\cup B_{D'}$ as defined in~(\ref{indeppart}). Let $u$ be a vertex in $A_{D'}$. By~(c), $\epn(u,D')\ne \emptyset$. First we show that $u\in D$. Suppose, to the contrary, that $u \notin D$. Since $D$ is a dominating set of $G$, every vertex in $V(G)\setminus D$ is adjacent to a vertex in $D$. Moreover since $D \subseteq D'$, we have $V(G) \setminus D' \subseteq V(G)\setminus D$. This implies that every vertex $x$ in $\epn(u,D')$ is adjacent to a vertex in $D$ that is different than $u$ (since, by supposition, $u \notin D$), contradicting the fact that $\epn(u,D')\ne \emptyset$. Hence, $u \in D$. Since $u \in A_{D'}$, by the definition of $A_{D'}$, there exists a neighbor of $u$ in $V(G)\setminus D'$, say $w$, that is not $D'$-defended. Thus, by part~(d), we have $\epn(u,D') \setminus N(w) \ne \emptyset$. Since $D\subseteq D'$, we have $w\in V(G)\setminus D'\subseteq V(G)\setminus D$ and $\epn(u,D') \subseteq \epn(u,D)$, and so $\epn(u,D) \setminus N(w) \ne \emptyset$, implying that $w$ is not $D$-defended by~$u$. Thus, $u \in A_D$. Since $u$ is an arbitrary vertex in $A_{D'}$, we infer that $A_{D'} \subseteq A_D$. Thus since $D\subseteq D'$, we have $B_D = D \setminus A_D \subseteq D' \setminus A_D \subseteq D'\setminus A_{D'} = B_{D'}$, that is, $B_D\subseteq B_{D'}$.
\end{proof}

Let $I$ is an $\alpha$-set of $G$, and so $I$ is a maximum independent set in $G$. Further, let $u\in I$ such that $\epn(u,I) \ne \emptyset$.  If there exist two vertices $x,y\in \epn(u,I)$ such that $xy\notin E(G)$, then $(I\setminus \{u\})\cup \{x,y\}$ is an independent set larger than $I$, a contradiction. Hence, $\epn(u,I)$ is necessarily a clique. This fact has appeared implicitly in several papers  (see, for example,~\cite{degawa23, merouane15}). We state this property of a maximum independent set formally as follows.

\begin{lemma}[\cite{degawa23, merouane15}]	
\label{epncomp}	
If $I$ is an $\alpha$-set of a graph $G$, then for every $u \in I$ either $\epn(u,I) = \emptyset$ or $\epn(u,I) \ne \emptyset$ and $\epn(u,I)$ is a clique.
\end{lemma}

%

\section{$P_5$-free graphs}
\label{Sect:P5-free}

Our aim in this section is to prove Theorem~\ref{thm:main1}, that is, we prove that if $G$ is a $P_5$-free graph, then $\gamma_s(G) \le \frac{3}{2}\alpha(G)$. We note that this bound can be achieved on the disjoint union of $C_5$'s. Let $G$ be a $P_5$-free graph (possibly disconnected) and $I$ be an $\alpha$-set of $G$. By Observation~\ref{dom}, the set $I$ is a dominating set of $G$. If $|I|=1$, then $G$ is a complete graph and hence $I$ is a secure dominating set of $G$, and so in this case $\gamma_s(G) = 1 = \alpha(G)$. If $|I| \ge 2$ and $I$ is a secure dominating set of $G$, then we have $\gamma_s(G) \le |I| = \alpha(G)$. Hence we may assume that $|I| \ge 2$ and $I$ is not a secure dominating set of $G$, for otherwise, $\gamma_s(G) \le \alpha(G) < \frac{3}{2}\alpha(G)$, as desired. We now describe a procedure that we call ``\sdsp'' that builds a secure dominating set $S$ of $G$ starting with the initial $\alpha$-set $I$ and such that at each stage of the procedure the property that the resulting set, which includes vertices from $V(G) \setminus I$ iteratively to $S$, is a dominating set of $G$ is preserved. For every iteration of the procedure \sdsp, we partition the set $S$ as $A_S\cup B_S$ as defined in (\ref{indeppart}).

\begin{procedure}
\caption{SDS-$P_5$-FREE()}\label{process}
\KwIn{A $P_5$-free graph $G$ and a maximum independent set $I$ of $G$.}
\KwOut{A secure dominating set $S$ of $G$ such that $|S|\le \frac{3}{2}|I|$.}
Initially, set $S=I$ and $i=2$\;
		\While{$i\le |I|$}{
			\eIf{there exists a vertex $v\in V(G)\setminus S$ such that $v$ is not $S$-defended and has exactly $i$ neighbors in $S$}{
				for a neighbor $u$ of $v$ in $S$, pick a vertex $x\in \epn(u,S)\setminus N(v)$\;
				$S=S\cup \{x\}$\;
			}{
				$i=i+1$\;
			}
		}
		\Return{$S$};
\end{procedure}
	
\medskip
\noindent{\bf \large Underlying Ideas:}

The procedure \sdsp iteratively constructs a nested sequence of supersets of $S$ which was initially the $\alpha$-set $I$. The nested sequence of supersets of $S$ is constructed by iteratively including some vertices from $V(G)\setminus S$ into $S$ . Since a superset of a dominating set is also a dominating set, $S$ remains a dominating set of $G$ throughout the procedure. The procedure finds a vertex $v\in V(G)\setminus S$ that is not $S$-defended and has exactly $i$ neighbors in $S$ in Step $3$. First assume that such a vertex~$v$ exists. Since $S$ is a dominating set of $G$, we have $N(v)\cap S\ne  \emptyset$. This ensures the existence of the vertex $u$ in $N(v)\cap S$ in Step~4. By Lemma~\ref{mainthmlem-1}(b), $N(v)\cap S\subseteq A_S$ and hence $u\in N(v)\cap A_S$.  By Lemma~\ref{mainthmlem-1}(d), we infer that $\epn(u,S)\setminus N(v)\ne \emptyset$. This ensures the existence of the vertex $x$ in $\epn(u,S)\setminus N(v)$ in Step $4$. The procedure includes $x$ into $S$ in Step~$5$. We will prove that by the inclusion of $x$ into $S$,  the vertex $v$ and possibly some more vertices in $V(G)\setminus S$ are $S$-defended and the cardinality of $A_S$ is decreased by at least $2$. The above process is repeated until every vertex in $V(G)\setminus S$ that has exactly $i$ neighbors in $S$ is $S$-defended. Now we may assume that the procedure cannot find a vertex $v$ in Step~3. At this stage of the procedure, the value of $i$ is incremented by~$1$ and the above process is repeated. In this way, at the end of the procedure, each vertex in $V(G)\setminus S$ is  $S$-defended and hence $S$ is a secure dominating set of $G$. Finally we prove that upon completion of the procedure, the resulting set $S$ satisfies $|S| \le \frac{3}{2}|I| = \frac{3}{2}\alpha(G)$. We note that one iteration of the above process either updates $S$ to a proper superset or increases the value of $i$ by~$1$.

Before presenting a series of lemmas, we discuss a problematic situation which may arise during the implementation of procedure \sdsp. Let $z$ be a vertex in $V(G)\setminus S$ such that $z$ is not $S$-defended for a maintained set $S$ and has at least $i$ neighbors in $S$. Suppose that $z$ is not selected as the chosen vertex~$v$ in Step $3$ for some iterations in the procedure \sdsp. Further suppose that after those iterations, the vertex $z$ has less than~$i$ neighbors in the updated set $S$ and $z$ is still not $S$-defended. In that case, the procedure may skip $z$ and end up without settling all the remaining vertices that are not $S$-defended. We will show that no such situation occurs by proving that if a vertex $z$ is not $S$-defended for the maintained set $S$, then $|N(z) \cap S| = |N(z)\cap \widetilde{S}|$ for the maintained set $\widetilde{S}$ of any earlier iteration of the procedure.
	
We now present a series of lemmas that will prove helpful when proving Theorem~\ref{thm:main1}.

\begin{lemma}\label{mainthmcl-1}
If $S$ is a dominating set of $G$ and a vertex $v\in V(G) \setminus S$ is not $S$-defended, then $v$ has at least two neighbors in $A_S$.
\end{lemma}
\begin{proof}
Let $S$ be a dominating set of $G$, and let $v$ be a vertex in $V(G) \setminus S$ that is not $S$-defended. By Lemma~\ref{mainthmlem-1}(b), $N(v)\cap S\subseteq A_S$. Also, since $S$ is a dominating set of $G$, the vertex $v$ has at least one neighbor in $S$ and hence in $A_S$. Suppose, to the contrary, that $v$ has exactly one neighbor in $A_S$, say $w$. Since $N(v)\cap S\subseteq A_S$, we have $N(v)\cap S=\{w\}$, that is $v\in \epn(w,S)$. Since $I\subseteq S$, by Lemma~\ref{mainthmlem-1}(e), we have $w\in A_S\subseteq A_I\subseteq I$. By Lemma~\ref{mainthmlem-1}(d), there exists a vertex $w'\in \epn(w,S)$ that is not adjacent to $v$. Now since $I \subseteq S$ and $w\in I$, we have $v,w' \in \epn(w,I)$. However, $I$ is an $\alpha$-set of $G$, and so by Lemma~\ref{epncomp}, we infer that $\epn(w,I)$ is a clique. In particular, $v$ and $w'$ are adjacent, a contradiction. Hence, $v$ has at least two neighbors in $A_S$.
\end{proof}
	
\begin{lemma}\label{mainthmcl-2}
Consider any iteration of the procedure \sdsp. Let $S$ and $S'$ be the maintained sets at the beginning and at the end of an iteration of the procedure \sdsp, respectively.  If a vertex $v\in V(G)\setminus S'$ is not $S'$-defended, then $v$ is not $S$-defended. Moreover, $|N(v)\cap S|=|N(v)\cap S'|$.
\end{lemma}
\begin{proof}
Recall that $S\subseteq S'$. Let $v$ be a vertex in $V(G)\setminus S'$ such that $v$ is not $S'$-defended and $v$ has $k$ neighbors in $S'$. Since $S\subseteq S'$, we have $v\in V(G)\setminus S$. Suppose, to the contrary, that the vertex $v$ is $S$-defended. Then there exists a vertex $u\in S$ such that $(S\setminus \{u\})\cup\{v\}$ is a dominating set. Then, since $S'$ is a dominating set of $G$ and $u\in S\subseteq S'$, the set $(S'\setminus \{u\})\cup\{v\}$ is a dominating set of $G$, and so the vertex $v$ is $S'$-defended, a contradiction. Hence, the vertex $v$ is not $S$-defended. Since $S\subseteq S'$, we have $|N(v)\cap S|\le k$. Now we show that $|N(v)\cap S|\ge k$. Since $v$ is not $S'$-defended, by Lemma~\ref{mainthmlem-1}(b), we have $N(v)\cap S'\subseteq A_{S'}$. Hence all the $k$ neighbors of $v$ in $S'$ are in $A_{S'}$. By Lemma~\ref{mainthmlem-1}(e), $A_{S'}\subseteq A_S$ since $S\subseteq S'$. Thus, $|N(v)\cap A_S|\ge k$. Consequently, $|N(v)\cap S|\ge k$, implying that $|N(v)\cap S|= k=|N(v)\cap S'|$.
\end{proof}
	
We note that the set $S$ considered in Lemma~\ref{mainthmcl-2} is also the set obtained at the end of the previous iteration. So by repeatedly applications of Lemma~\ref{mainthmcl-2}, we infer that if a vertex $v$ is not $S$-defended for the maintained set $S$ of any iteration of the procedure, then $v$ is not $\widetilde{S}$-defended for any maintained set $\widetilde{S}$ of any previous iteration. This yields the following corollary of Lemma~\ref{mainthmcl-2}.
	
\begin{coro}
\label{mainthmcl-2-coro}
Consider any iteration of the procedure \sdsp. Let $S$ be the maintained set at the beginning of the iteration of the procedure \sdsp.  If a vertex $v\in V(G)\setminus S$ is not $S$-defended, then $v$ is not $\widetilde{S}$-defended for any maintained set $\widetilde{S}$ of any previous iteration of the procedure. Moreover, $|N(v)\cap S|=|N(v)\cap \widetilde{S}|$. In particular,  $|N(v)\cap S|=|N(v)\cap I|$.
\end{coro}
	
\begin{lemma}\label{mainthmcl-2.1}
Consider an iteration of the procedure \sdsp. Let $S$ be the maintained set at the beginning of the iteration of the procedure \sdsp, and let $v$, $u$ and $x$ be the vertices obtained in Steps~$3$-$4$. If $S'$ is the proper superset of $S$ obtained at the end of the iteration of the procedure \sdsp, then $u,x\in B_{S'}$ and $v$ is $S'$-defended.
\end{lemma}
\begin{proof}
For the current value of $i$, the vertex $v$ is not $S$-defended and $v$ has exactly $i$ neighbors in $S$. Recall that $u\in N(v)\cap S$, $x\in \epn(u,S)\setminus N(v)$, and $S'=S\cup\{x\}$. By Lemma~\ref{mainthmlem-1}(b), $N(v) \cap S\subseteq A_S$, and so $u \in N(v)\cap A_S$. First we show that $\epn(u,S)$ is a clique. Suppose, to the contrary, that there exist vertices $a,b \in \epn(u,S)$ that are not adjacent. Since $A_S\subseteq A_I\subseteq I$, we note that $u \in I$. Also, since $a,b\in \epn(u,S)$, the vertices $a$ and $b$ do not have any neighbor in $S\setminus \{u\}$. Hence, $a$ and $b$ do not have any neighbor in $I\setminus \{u\}$ since $I\subseteq S$. Thus, $(I\setminus \{u\})\cup \{a,b\}$ is an independent set larger than $I$, contradicting the maximality of the set $I$. Hence, $\epn(u,S)$ is a clique, and therefore $\epn(u,S) \cup \{u\}$ is a clique. Now since $S'=S\cup \{x\}$, we have $\epn(u,S') = \epn(x,S') = \emptyset$. Thus by Lemma~\ref{mainthmlem-1}(c), we infer that $u,x \notin A_{S'}$. Since $S'=A_{S'}\cup B_{S'}$, we have $u,x\in B_{S'}$. The vertex $v$ is therefore adjacent to a vertex, namely $u$, in $B_{S'}$, and so by Lemma~\ref{mainthmlem-1}(b), $v$ is $S'$-defended.
\end{proof}

\begin{lemma}\label{mainthmcl-3}
Consider an iteration of the procedure \sdsp. Let $S$ be the maintained set at the beginning of the iteration of the procedure \sdsp, and let $v$, $u$ and $x$ be the vertices obtained in Steps~$3$-$4$. If $S'$ is the proper superset of $S$ obtained at the end of the iteration of the procedure \sdsp, then $A_{S'} \subseteq A_S \setminus (N(v)\cap S)$.
\end{lemma}
\begin{proof}
Recall that for the current value of $i$, the vertex $v$ is not $S$-defended and $v$ has exactly $i$ neighbors in $S$. We also recall that $u\in N(v)\cap S$, $x\in \epn(u,S)\setminus N(v)$, and $S' = S\cup\{x\}$.  Let $y$ be an arbitrary vertex in $A_{S'}$. To prove the lemma, it is sufficient to show that $y\in A_S\setminus (N(v)\cap S)$. By Lemma~\ref{mainthmlem-1}(e), $A_{S'}\subseteq A_S$ and hence $y\in A_S$. Suppose, to the contrary, that $y\in N(v)\cap S$. Since $y\in A_{S'}$, by the definition of $A_{S'}$, there exists a neighbor of $y$ in $V(G)\setminus S'$, say $z$, that is not $S'$-defended. By Lemma~\ref{mainthmcl-2.1}, the vertex $v$ is $S'$-defended and $u,x\in B_{S'}$, implying that $z \ne  v$. Also, by Lemma~\ref{mainthmlem-1}(b), the vertex $z$ is adjacent to neither $u$ nor $x$.

\begin{claim}\label{lem6-cl1}
The vertex $z$ has at least $i$ neighbors in $S'$.
\end{claim}
\begin{proof}[Proof of Claim~\ref{lem6-cl1}]
By Corollary~\ref{mainthmcl-2-coro}, the vertex $z$ is not $\widetilde{S}$-defended and $|N(z)\cap S'|=|N(z)\cap S|=|N(z)\cap \widetilde{S}|$ for any maintained set $\widetilde{S}$ of any previous iteration in the procedure SDS $P_5$-FREE. For the current value of $i$, if $z$ has less than $i$ neighbors in $S'$, then it must have been selected as the vertex $v$ by the procedure in Step~3 in some previous iteration. Let $\widetilde{S'}$ be the maintained set at the end of that iteration. By Lemma~\ref{mainthmcl-2.1}, the vertex $z$ is $\widetilde{S'}$-defended. Since $\widetilde{S'} \subseteq S'$, the vertex $z$ is therefore $S'$-defended, a contradiction. Hence, $z$ has at least $i$ neighbors in $S'$, as claimed.
\end{proof}

\begin{claim}\label{lem6-cl2}
The vertex $z$ has a neighbor in $A_{S'}$ that is not adjacent to $v$.
\end{claim}
\begin{proof}[Proof of Claim~\ref{lem6-cl2}]
Suppose, to the contrary, that $v$ is adjacent to every vertex in $N(z)\cap A_{S'}$. By our earlier observations, the vertex $v$ has exactly $i$ neighbors in $S$, $u \in N(v)\cap A_S$, and $x\in \epn(u,S)\setminus N(v)$. Also, recall that $u,x\in B_{S'}$. By Claim~\ref{lem6-cl1}, the vertex $z$ has at least $i$ neighbors in $S'$. Since $z$ is not $S'$-defended, by Lemma~\ref{mainthmlem-1}(b), we have $N(z)\cap S'\subseteq A_{S'}$. Hence, $z$ has at least $i$ neighbors in $A_{S'}$. Since $v$ is adjacent to every vertex of $N(z)\cap A_{S'}$, the vertex $v$ therefore has at least $i$ neighbors in $A_{S'}$. By Lemma~\ref{mainthmlem-1}(e) and the fact that $u\in B_{S'}$, we have $A_{S'}\subseteq A_S\setminus \{u\}$. Hence, $v$ has at least $i$ neighbors in $A_S\setminus \{u\}$. However since $v$ is adjacent to $u\in S$, the vertex $v$ therefore has more than $i$ neighbors in $S$, a contradiction. Hence, $z$ has a neighbor in $A_{S'}$ that is not adjacent to $v$, as claimed.
\end{proof}
	
We now return to the proof of the lemma. Recall that $z$ is a neighbor of $y$ in $V(G)\setminus S'$ that is not $S'$-defended. By Claim~\ref{lem6-cl2}, the vertex $z$ has a neighbor in $A_{S'}$, say $w$, that is not adjacent to $v$. Since $S\subseteq S'$, by Lemma~\ref{mainthmlem-1}(e), $A_{S'}\subseteq A_S$ and hence $w\in A_S$. By our earlier observations, $\{u,w,y\} \subseteq A_S$. Also, recall that $I\subseteq S$ and $I$ is an $\alpha$-set of $G$. By Lemma~\ref{mainthmlem-1}(e), $A_{S}\subseteq A_I$, and so $\{u,w,y\} \subseteq A_I \subseteq I$ and hence $u,w$, and $y$ are mutually non-adjacent vertices. Since $x\in \epn(u,S)\setminus N(v)$, the vertex $x$ is not adjacent to $v$, $w$ and $y$. Recall that $z$ is not adjacent to $u$ and $x$. If $v$ is not adjacent to $z$, then $\{x,u,v,y,z\}$ induces a $P_5$ in $G$, a contradiction. Hence, $v$ is adjacent to $z$. However, then $\{x,u,v,z,w\}$ induces a $P_5$ in $G$, a contradiction. Hence we conclude that $y \notin N(v)\cap S$. By our earlier observations,  $y \in A_{S'} \subseteq A_S$. Therefore, $y \in A_S\setminus (N(v)\cap S)$. This completes the proof of Lemma~\ref{mainthmcl-3}.	
\end{proof}

We are now in a position to present a proof of Theorem~\ref{thm:main1}. Recall its statement.

\medskip
\noindent \textbf{Theorem~\ref{thm:main1}}. \emph{If $G$ is a $P_5$-free graph, then $\gamma_s(G) \le \frac{3}{2}\alpha(G)$.
}

\begin{proof}
Let $I$ be an $\alpha$-set of $G$. By Observation~\ref{dom}, $I$ is a dominating set of $G$. If $I$ is a secure dominating set of $G$, then $\gamma_s(G)\le |I|=\alpha(G)\le \frac{3}{2}\alpha(G)$. Hence we may assume that $I$ is not a secure dominating set of $G$, for otherwise the desired result follows immediately. Hence there exists a vertex in $V(G)\setminus I$ that is not $I$-defended. We now apply the procedure SDS $P_5$-FREE. The procedure constructs a nested sequence of supersets of $I$. As discussed in the underlying idea of the procedure, it is sufficient to show that for the set $S^*$ obtained at the end of the procedure, $S^*$ is a secure dominating set of $G$ and $|S^*|\le \frac{3}{2}\alpha(G)$. 	
	
We first prove that $S^*$ is a secure dominating set of $G$. To achieve this, we show that at every iteration if there is a vertex $u\in V(G)\setminus S$ that is not $S$-defended with respect to the maintained set $S$ in the procedure, then the procedure adds to the set an appropriate vertex so that the vertex $u$ is $S$-defended with respect to the updated set $S$ at the end of the iteration.

Let $w$ be a vertex in $V(G)\setminus S$ that is not $S$-defended. We note that in the procedure \sdsp, the value of $i$ increases from $2$ to $|I|$ during the execution of the procedure. By Lemma~\ref{mainthmcl-1}, the vertex $w$ has at least two neighbors in $S$. This justifies the minimum value of $i$ to be $2$. By Corollary~\ref{mainthmcl-2-coro}, $|N(w)\cap S|=|N(w)\cap I|$. Hence, $w$ has at most $|I|$ neighbors in $S$. This justifies the upper limit of $i$ to be $|I|$. The procedure \sdsp~iteratively finds a vertex $v$ in $V(G)\setminus S$ which is not $S$-defended and has exactly $i$ neighbors in $S$ in Step~3. If it does not find such a vertex, then the value of $i$ is incremented by~$1$. Suppose, to the contrary, that such a vertex $v$ exists. Then by Lemma~\ref{mainthmcl-2.1}, at the end of the iteration, the maintained set $S$ is updated to a proper superset, say $S'$, such that $v$ (and possibly some more vertices in $V(G)\setminus S$) are $S'$-defended.

We now consider a vertex $z$ that is not $S$-defended. We show that $z$ has at least $i$ neighbors in~$S$. Suppose, to the contrary, that $z$ has $k$ neighbors in $S$ and $k<i$. By Corollary~\ref{mainthmcl-2-coro}, the vertex $z$ has exactly $k$ neighbors in $\widetilde{S}$ for any maintained set $\widetilde{S}$ of any previous iteration of the procedure. Thus, $z$ must have been selected as the vertex $v$ in Step~3 in an earlier iteration of the procedure. Let $\widetilde{S'}$ be the set obtained at the end of that iteration. By Lemma~\ref{mainthmcl-2.1}, the vertex $z$ is $\widetilde{S'}$-defended. Since $\widetilde{S'}\subseteq S$, the vertex $z$ is $S$-defended, a contradiction. Hence, $z$ has at least $i$ neighbors in $S$.  Therefore, if $S^*$ is the superset of $I$ at the end of the procedure, then every vertex in $V(G)\setminus I$ which is not $I$-defended becomes $S^*$-defended. Hence $S^*$ is a secure dominating set of $G$.

We prove next that $|S^*|\le \frac{3}{2}\alpha(G)$. Consider any iteration of the procedure \sdsp~in which the maintained set $S$ is updated to a proper superset $S'$. Let $v$, $u$, and $x$ be the vertices as described in the procedure in Steps~3-4, and so $S'=S\cup \{x\}$. By Lemma~\ref{mainthmcl-1}, $|N(v)\cap S| \ge 2$. By Lemma~\ref{mainthmcl-3}, $A_{S'}\subseteq A_S\setminus (N(v)\cap S)$, and so the vertices in $N(v) \cap S$ are removed from the set $A_S$ when constructing the reduced set $A_{S'}$. Thus, the cardinality of $A_S$ decreases by at least~$2$ when forming $A_{S'}$.  Specifically, $|A_S|$ decreases by at least~$2$ whenever $|S|$ is increased by $1$. Since $I\subseteq S$, by Lemma~\ref{mainthmlem-1}(e), we have $A_S\subseteq A_I$. Hence by our earlier observations, $|A_S|\le |A_I|\le |I|$. Since $A_{S^*}=\emptyset$ for the set $S^*$ obtained at the end of procedure and the procedure starts with $S=I$, we have $|S^*|\le |I|+\frac{1}{2}|A_{I}|\le |I|+\frac{1}{2}|I|=\frac{3}{2}|I|$. Hence $|S^*|\le \frac{3}{2}\alpha(G)$. This completes the proof of the theorem.
\end{proof}	

\section{$(P_3\cup P_2)$-free graphs}
\label{Sect:P3P2-free}

We note that the class of $(P_3 \cup P_2)$-free graphs is a subclass of the class of $P_6$-free graphs. In this section, we prove Theorems~\ref{thm:main2} and~\ref{thm:main3}. Recall the statement of Theorem~\ref{thm:main2}.

\medskip
\noindent \textbf{Theorem~\ref{thm:main2}}. \emph{If $G$ is a $(P_3 \cup P_2)$-free graph, then $\gamma_s(G)\le \alpha(G)+1$.
}

\begin{proof}
Let $G$ be a $(P_3 \cup P_2)$-free graph (possibly disconnected). Let $I$ be an $\alpha$-set of $G$. By Observation~\ref{dom}, $I$ is a dominating set of $G$. We partition $I$ as $A_I\cup B_I$ as defined in (\ref{indeppart}). If $A_I=\emptyset$, then by Lemma~\ref{mainthmlem-1}(a), $I$ is a secure dominating set of $G$ implying that $\gamma_s(G)\le \alpha(G)$. Hence, we may assume that $A_I \ne  \emptyset$, for otherwise the desired result follows. Let $u$ be a vertex in $A_I$. By Lemma~\ref{mainthmlem-1}(c), $\epn(u,I)\ne  \emptyset$. Let $x$ be a vertex in $\epn(u,I)$ and consider the set $S=I\cup \{x\}$. Since $S$ is a superset of $I$, the set $S$ is a dominating set of $G$. If every vertex in $V(G)\setminus S$ is $S$-defended, then $S$ is a secure dominating set of $G$ implying that $\gamma_s(G) \le |S| = \alpha(G)+1$, as desired. Hence, we may assume that there exists a vertex $v\in V(G)\setminus S$ that is not $S$-defended. Let $A_S$ and $B_S$ be the sets as defined in (\ref{indeppart}). We now show that due to the existence of $v$, we get a contradiction to the fact that $G$ is $(P_3 \cup P_2)$-free.
		
Since $I$ is a dominating set of $G$ and $x\in V(G)\setminus I$, we have  $\epn(x,S)=\emptyset$, and so $x\notin A_S$. By Lemma~\ref{epncomp}, $\epn(u,I)$ is a clique of $G$. Hence, every vertex of $\epn(u,I)\setminus \{x\}$ is adjacent to $x$, implying that $\epn(u,S)=\emptyset$ and hence $u\notin A_S$. Then, since $S=A_S\cup B_S$, we have $x,u\in B_S$.  Since $v$ is not $S$-defended, by Lemma~\ref{mainthmlem-1}(b), we have $N(v)\cap S\subseteq A_S$ and hence $v$ is adjacent to neither $u$ nor~$x$.

We claim that $v$ has at least two neighbors in $A_S$. Suppose, to the contrary, that $v$ has only one neighbor in $A_S$, say $z$. By Lemma~\ref{mainthmlem-1}(b), $N(v)\cap S=\{z\}$, that is $v\in \epn(z,S)$. By Lemma~\ref{mainthmlem-1}(e), $A_S\subseteq A_I\subseteq I$ and hence $z\in I$. By Lemma \ref{mainthmlem-1}(d), there is a vertex in $\epn(z,S)$, say $y$, that is not adjacent to $v$. Now since $I\subseteq S$ and $z\in I$, we infer that $\{v,y\} \subseteq \epn(z,I)$. Since $I$ is a maximum independent set of $G$, by Lemma~\ref{epncomp}, $\epn(z,I)$ is a clique, a contradiction to the fact that $v$ is not adjacent to $y$. Hence, $v$ has at least two neighbors in $A_S$. Let $a,b \in N(v) \cap A_S$. By Lemma~\ref{mainthmlem-1}(e), $A_S\subseteq A_I\subseteq I$ and hence $a,b \in I$. Since $\{a,b,u\} \subseteq I$, the set $\{a,b,u\}$ is independent. Moreover, $x$ is adjacent to neither $a$ nor $b$ since $x \in \epn(u,I)$. By our earlier observations, $v$ is adjacent to neither $u$ nor $x$. The set $\{a,v,b,u,x\}$ therefore induces a $P_3\cup P_2$ in $G$, a contradiction. We therefore conclude that the vertex $v$ does not exist, implying that the set $S$ is a secure dominating set of $G$, and so $\gamma_s(G)\le |S| = \alpha(G)+1$. 		
\end{proof}

We now improve the bound given in Theorem~\ref{thm:main2} for the classes of $(P_3 \cup P_1)$-free graphs and $(K_2 \cup 2K_1)$-free graphs, both of which are subclasses of the class of $(P_3 \cup P_2)$-free graphs. In the following lemma, we obtain structural properties of $(P_3 \cup P_1)$-free graphs. As an application of this lemma, we prove that $\gamma_s(G)\le \max\{3,\alpha(G)\}$ if $G$ is a $(P_3 \cup P_1)$-free graph.

\begin{lemma}\label{p3up1lem}
If $G$ is a $(P_3 \cup P_1)$-free graph with $\alpha(G)\ge 3$ and $I$ be an $\alpha$-set of $G$, then every vertex in $V(G)\setminus I$ is either adjacent to exactly one vertex in $I$ or is adjacent to every vertex in $I$.
\end{lemma}
\begin{proof}
Let $v$ be an arbitrary vertex in $V(G)\setminus I$. By Observation~\ref{dom}, the set $I$ is dominating set of $G$ and hence $v$ has at least one neighbor in $I$. Let $u$ be an arbitrary neighbor of $v$ in $I$. If $u$ is the only neighbor of $v$ in $I$, then we are done. Hence we may assume that $v$ has at least two neighbors in $I$. Let $w$ be a neighbor of $v$ in $I$ different from~$u$. Suppose, to the contrary, that $v$ is not adjacent to every vertex in $I$, and let $z$ be a vertex in $I$ that is not adjacent to $v$. Since $I$ is an independent set and $\{u,w,z\} \subseteq I$, the set $\{u,w,z\}$ is an independent set. Thus, $\{u,v,w,z\}$ induces a $P_3\cup P_1$ in $G$, a contradiction. Hence, $v$ is adjacent to every vertex in $I$.
\end{proof}

\begin{theorem}\label{P3UP1}
If $G$ is a $(P_3 \cup P_1)$-free  graph, then $\gamma_s(G)\le \max\{3,\alpha(G)\}$.
\end{theorem}
\begin{proof}
Let $G$ be a $(P_3 \cup P_1)$-free graph (possibly disconnected) and let $I$ be an $\alpha$-set of $G$. By Observation~\ref{dom}, the set $I$ is a dominating set of $G$. First assume that $|I| \le 2$, that is $\alpha(G)\le 2$. Since the class of $(P_3 \cup P_1)$-free graphs is a subclasses of the class of $(P_3 \cup P_2)$-free graphs, the graph $G$ is $(P_3 \cup P_2)$-free. Hence by Theorem~\ref{thm:main2}, $\gamma_s(G) \le \alpha(G) + 1 \le 3$, yielding the desired result. Hence we may assume that $|I| \ge 3$, that is $\alpha(G)\ge 3$.

We show that $I$ is a secure dominating set of $G$. Suppose, to the contrary, that $I$ is not a secure dominating set of $G$. Let $v$ be a vertex that if not $I$-defended. We note that $v \in V(G)\setminus I$. By assumption, $|I| = \alpha(G) \ge 3$. By Lemma~\ref{p3up1lem}, the vertex $v$ is either adjacent to exactly one vertex in $I$ or adjacent to every vertex in $I$. Suppose that $v$ is adjacent to every vertex in $I$, and let $a$, $b$ and $c$ be three distinct vertices in $I$, and so $\{a,b,c\} \subseteq N(v) \cap I$. By Lemma~\ref{mainthmlem-1}(b), $\{a,b,c\} \subseteq A_I$. By Lemma~\ref{mainthmlem-1}(c)--(d), $\epn(a,I) \ne \emptyset$ and $\epn(a,I) \setminus N(v) \ne \emptyset$. Let $x \in \epn(a,I) \setminus N(v)$. In particular, $x$ is not adjacent to $b$, $c$ and $v$. Thus the set $\{b,v,c,x\}$ induces a $P_3\cup P_1$ in $G$, a contradiction. Hence, $v$ is adjacent to exactly one vertex in $I$, say $u$. Thus, $v \in \epn(u,I)$. By Lemma~\ref{epncomp}, $\epn(u,I)$ is a clique, implying that $(I \setminus \{u\}) \cup\{v\}$ is a dominating set of $G$, that is, the vertex $v$ is $I$-defended, a contradiction. Hence, $I$ is a secure dominating set of $G$, and so $\gamma_s(G) \le |I| = \alpha(G)$, once again yielding the desired result.
\end{proof}

In the following lemma, we obtain structural properties of $(K_2 \cup 2K_1)$-free graph. As an application of this lemma, we prove that $\gamma_s(G) \le \max\{3,\alpha(G)\}$ if $G$ is a $(K_2 \cup 2K_1)$-free graph.

\begin{lemma}\label{k2u2k1lem}
If $G$ is a $(K_2 \cup 2K_1)$-free graph with $\alpha(G)\ge 3$ and $I$ be an $\alpha$-set of $G$, then every vertex in $V(G)\setminus I$ is adjacent at least two vertices in $I$.
\end{lemma}
\begin{proof}
Let $v$ be an arbitrary vertex in $V(G)\setminus I$. By Observation~\ref{dom}, the set $I$ is a dominating set of $G$ and hence $v$ has at least one neighbor in $I$. Suppose, to the contrary, that $v$ has exactly one neighbor in $I$, say $u$. Since $\alpha(G) \ge 3$, the $\alpha$-set $I$ satisfies $|I| \ge 3$. Let $w$ and $z$ be two distinct vertices in $I$ other than $u$. Since $I$ is an independent set and $\{u,w,z\} \subseteq I$, the set $\{u,w,z\}$ is an independent set. Hence, $\{v,u,w,z\}$ induces a $K_2 \cup 2K_1$ in $G$, a contradiction. Therefore, the vertex $v$ has at least two neighbors in $I$.
\end{proof}

\begin{theorem}\label{K2U2K1}
If $G$ is a $(K_2 \cup 2K_1)$-free graph, then $\gamma_s(G) \le \max\{3,\alpha(G)\}$.
\end{theorem}
\begin{proof}
Let $G$ be a $(K_2 \cup 2K_1)$-free graph (possibly disconnected) and let $I$ be an $\alpha$-set of $G$. By Observation~\ref{dom}, the set $I$ is a dominating set of $G$. First assume that $|I| \le 2$, that is $\alpha(G)\le 2$. Since the class of $(K_2 \cup 2K_1)$-free graphs is a subclasses of the class of $(P_3 \cup P_2)$-free graphs, the graph $G$ is $(P_3 \cup P_2)$-free. Hence by Theorem~\ref{thm:main2}, $\gamma_s(G) \le \alpha(G) + 1 \le 3$, yielding the desired result. Hence we may assume that $|I| \ge 3$, that is $\alpha(G)\ge 3$.
We partition $I$ as $A_I\cup B_I$ as defined in~(\ref{indeppart}). By Lemma~\ref{k2u2k1lem}, every vertex $v$ in $V(G)\setminus I$ is adjacent to at least two vertices in $I$. This implies that $\epn(u,I) = \emptyset$ for every vertex $u$ in $I$. Hence by Lemma~\ref{mainthmlem-1}(c), we infer that $A_I = \emptyset$. Thus by Lemma~\ref{mainthmlem-1}(a), the $\alpha$-set $I$ of $G$ is a secure dominating set of $G$, and so $\gamma_s(G) \le |I| = \alpha(G)$, once again yielding the desired result.
\end{proof}

\section{$(P_5,\paw)$-free graphs}
\label{Sect:P5C3-free}

In this section, we improve the upper bound given in  Theorem~\ref{thm:main1} on $(P_5, paw)$-free graphs, a subclass of the class of $P_5$-free graphs, by using structural properties of such graphs. We shall prove Theorem~\ref{thm:main4} which states that if $G$ is a connected $(P_5,\paw)$-free graph, then $\gamma_s(G) \le \max\{3,\alpha(G)\}$. In order to prove this result, we first prove the following key lemma.

\begin{lemma}\label{c3cl-1}
If $G$ is a connected $(P_5,C_3)$-free graph and $C \colon u_1u_2u_3u_4u_5u_1$ is an induced $C_5$ in $G$, then every vertex in $V(G)\setminus V(C)$ has exactly two non-consecutive neighbors in $V(C)$.
\end{lemma}
\begin{proof}
Let $v\in V(G)\setminus V(C)$. We first show that if $v$ has a neighbor in $G$ that belongs to the cycle $C$, then it has exactly two neighbors in $V(C)$ and these neighbors are non-consecutive. Hence assume that $v$ has a neighbor in $V(C)$. If $v$ has at least three neighbors in $V(C)$, then $G[V(C)\cup \{v\}]$ contains a $C_3$, a contradiction. Hence, the vertex $v$ has either one or two neighbors on $V(C)$. First assume that $|N(v)\cap V(C)|= 1$. Renaming vertices if necessary, we may assume by symmetry that $N(v)\cap V(C)=\{u_1\}$. Then $\{v,u_1,u_2,u_3,u_4\}$ induces a $P_5$ in $G$, a contradiction. Therefore, $|N(v)\cap V(C)|=2$. If the neighbors of $v$ in $V(C)$ are consecutive vertices of $C$, then $G[V(C)\cup \{v\}]$ contains a $C_3$, a contradiction. Hence the neighbors of $v$ in $V(C)$ are non-consecutive vertices of $C$. Thus, every vertex not on the cycle $C$ that has a neighbor in $V(C)$, has exactly two neighbors in $V(C)$ and these neighbors are non-consecutive.

We show next that every vertex not on the cycle $C$ has a neighbor in $G$ that belongs to the cycle $C$. Suppose, to the contrary, that the cycle $C$ is not a dominating cycle, that is, suppose that the set $V(C)$ is not a dominating set of $G$. By the connectivity of $G$, there exists a vertex $v_2$ that is not dominated by the set $V(C)$ but is at distance~$2$ from some vertex of $C$. Renaming vertices of $C$, if necessary, we may assume that $u_1$ and $v_2$ are at distance~$2$ in $G$. Let $v_1$ be a common neighbor of $u_1$ and $v_2$ in $G$. Thus, $v_1$ does not belong to the cycle $C$ but has a neighbor in $G$ that belongs to the cycle $C$, namely $u_1$. By our earlier observations, the vertex $v_1$ has exactly two neighbors in $V(C)$ and these neighbors are non-consecutive. By symmetry, we may assume that $v_1$ is adjacent to $u_1$ and $u_3$. However, then, $\{u_4,u_5,u_1,v_1,v_2\}$ induces a $P_5$ in $G$, a contradiction. Therefore, every vertex in $V(G) \setminus V(C)$ has a neighbor in $V(C)$. By our earlier observations, we infer that every vertex in $V(G)\setminus V(C)$ has exactly two non-consecutive neighbors in $V(C)$.
\end{proof}

Let $G$ be a connected $(P_5,C_3)$-free graph and let $C \colon u_1u_2u_3u_4u_5u_1$ be an induced $C_5$ in $G$. If $G = C$, then $\gamma_s(G) = 3 = \max\{3,\alpha(G)\}$. Hence, we may assume that $G \ne C$, and so $V(G)\setminus V(C) \ne \emptyset$. By Lemma~\ref{c3cl-1}, we can partition $V(G)$ as $(U_1,U_2,U_3,U_4,U_5)$ where $U_i$ is defined as follows.
 \begin{align*}\tag{$\mathcal{P}_2$}\label{c3part}
 	U_i = \{u_i\} \cup \{v\in V(G)\setminus V(C) \, \colon  N(v)\cap V(C)=\{u_{i-1},u_{i+1}\}\}, \, i\in [5].
 \end{align*}

Since $u_i\in U_i$, we note that $U_i \ne \emptyset$ for every $i\in [5]$. In what follows, we use modulo~$5$ arithmetic for the indices. We proceed further with a series of useful claims.	

\begin{claim}\label{cl(b)}
For every $i\in [5]$, $U_i$ is an independent set, $[U_i, U_{i+1}]$ is complete, and $[U_i, U_{i+2}]=\emptyset$.
\end{claim}
\begin{proof} [Proof of Claim~\ref{cl(b)}]
If $a$ and $b$ are two adjacent vertices in $U_i$, then $\{a,b,u_{i+1}\}$ induces a $C_3$, a contradiction. Hence, $U_i$ is an independent set of $G$ for every $i \in [5]$. For some $x\in U_i$ and $y\in U_{i+1}$, if $xy\notin E(G)$, then  $\{y,u_{i+2},u_{i+3},u_{i+4}, x\}$ induces a $P_5$ in $G$, a contradiction. Hence $[U_i, U_{i+1}]$ is complete for every $i\in [5]$. For some $x\in U_i$ and $y\in U_{i+2}$, if $xy \in E(G)$, then $\{x, u_{i+1}, y\}$ induces a $C_3$ in $G$, a contradiction. Hence, $[U_i, U_{i+2}]=\emptyset$ for every $i\in [5]$.
\end{proof}

By Claim~\ref{cl(b)}, we infer that the graph $G$ is the expansion of a $5$-cycle in the sense that we replace each vertex $u_i$ on the $5$-cycle $C \colon u_1u_2u_3u_4u_5u_1$ with an independent set $U_i$ and where the open neighborhood of every vertex $v \in U_i$ in $G$ is given by $N(v) = U_{i-1} \cup U_{i+1}$ for all $i \in [5]$.

\begin{claim}\label{cl(c)}
$\gamma_s(G) \le 5$.
\end{claim}
\begin{proof}[Proof of Claim~\ref{cl(c)}]
Let $S = V(C) =\{u_1,u_2,u_3,u_4,u_5\}$. By the structure of the graph $G$ the set $S$ is a dominating set of $G$. Moreover, if $v$ is an arbitrary vertex in $V(G) \setminus S$, then $v \in U_i \setminus \{u_i\}$ for some $i \in [5]$, and so $N(v) \cap S = \{u_{i-1},u_{i+1}\}$. This implies that $\epn(u_i,S) = \emptyset$ for all $i \in [5]$. Hence by Lemma~\ref{mainthmlem-1}(c), $A_S = \emptyset$. Thus by Lemma~\ref{mainthmlem-1}(a), the set $S$ is a secure dominating set of $G$, and so $\gamma_s(G) \le |S| = 5$.
\end{proof}

\begin{claim}\label{cl(d)}
If $\alpha(G)= 4$, then $\gamma_s(G)\le 4$.
\end{claim}
\begin{proof}[Proof of Claim~\ref{cl(d)}]
If $|U_i|\ge 4$ for some $i\in [5]$, then by Claim~\ref{cl(b)}, the set $U_i\cup U_{i+2}$ is an independent set of size at least~$5$, a contradiction. Hence, $|U_i| \le 3$ for every $i\in [5]$. First we assume that $|U_i|\le 2$ for every $i\in[5]$. Let $S=\{u_1, u_2,u_3,u_4\}$. By our earlier observations, $S$ is a dominating set of $G$. If $|U_1| = 2$, then the vertex in $U_1\setminus \{u_1\}$ is $S$-defended by $u_2$. If $|U_2| = 2$, then the vertex in $U_2\setminus \{u_2\}$ is $S$-defended by $u_1$. If $|U_3| = 2$, then the vertex in $U_3 \setminus \{u_3\}$ is $S$-defended by $u_4$. If $|U_4| = 2$, then the vertex in $U_4\setminus \{u_4\}$ is $S$-defended by $u_3$. Finally, every vertex in $U_5$ is $S$-defended by $u_{1}$. Hence in this case when $|U_i|\le 2$ for every $i\in[5]$, every vertex in $V(G) \setminus S$ is $S$-defended, implying that $S$ is a secure dominating set of $G$, and so $\gamma_s(G) \le |S| = 4 = \alpha(G)$.
	
Hence we may assume that $|U_i|=3$ for some $i\in [5]$. Renaming vertices on the cycle $C$ if necessary, we may assume that $|U_2|=3$. If $|U_4| \ge 2$, then  $U_2 \cup U_4$ is an independent set of size at least~$5$, and if $|U_5| \ge 2$, then  $U_2 \cup U_5$ is an independent set of size at least~$5$. Both cases contradict the supposition that $\alpha(G) = 4$. Hence, $U_4=\{u_4\}$ and $U_5=\{u_5\}$. Since $U_1 \cup U_3$ is an independent set and $\alpha(G)= 4$, we infer that $|U_1| \le 2$ or $|U_3| \le 2$. Without loss of generality, we may assume that $|U_1| \le 2$. As before, we let $S=\{u_1,u_2,u_3,u_4\}$. We note that $S$ is a dominating set of $G$.  For $i \in [3]$, every vertex in $U_i \setminus \{u_i\}$ is $S$-defended by $u_{i+1}$. Moreover, the vertex $v_5$ is $S$-defended by $u_4$. Hence every vertex in $V(G)\setminus S$ is $S$-defended, implying as before that $S$ is a secure dominating set of $G$, and so $\gamma_s(G) \le |S| = 4 = \alpha(G)$.
\end{proof}

We are now in a position to prove the following upper bound on the secure domination number of a $(P_5, C_3)$-free graph.

\begin{theorem}\label{p5c3thm}
If $G$ is a connected $(P_5, C_3)$-free graph, then $\gamma_s(G)\le \max\{3,\alpha(G)\}$.
\end{theorem}
\begin{proof}
	Let $G$ be a connected $(P_5, C_3)$-free graph. If $G$ is a $C_5$-free graph, then, by Theorem \ref{cycle}, $\gamma_s(G)\le \alpha(G)$. Hence we assume that $G$ contains an induced $C_5$, for otherwise the desired upper bound follows. Let $C \colon u_1u_2u_3u_4u_5$ be an induced $C_5$ in $G$. We note that $\alpha(G) \ge 2$. If $G = C$, then $\gamma_s(G) = 3 = \max\{3,\alpha(G)\}$. Hence, we may assume that $G \ne C$, and so $V(G)\setminus V(C) \ne \emptyset$. By Lemma~\ref{c3cl-1}, we can partition $V(G)$ as $(U_1,U_2,U_3,U_4,U_5)$ as defined in (\ref{c3part}). Recall that $U_i \ne \emptyset$ for every $i\in [5]$. If $\alpha(G) \ge 4$, then by Claims~\ref{cl(c)}-\ref{cl(d)} we have $\gamma_s(G) \le \alpha(G)$, yielding the desired result.

Hence, we may assume that $\alpha(G) \le 3$. By Claim~\ref{cl(b)}, we infer that the graph $G$ is the expansion of a $5$-cycle in the sense that we replace each vertex $u_i$ on the $5$-cycle $C$ with an independent set $U_i$ and where the open neighborhood of every vertex $v \in U_i$ in $G$ is given by $N(v) = U_{i-1} \cup U_{i+1}$ for all $i \in [5]$. If $|U_i| \ge 3$ for some $i \in [5]$, then $U_i \cup U_{i+2}$ is an independent set of size at least~$4$, a contradiction. Hence, $|U_i| \le 2$ for all $i \in [5]$.

By our earlier assumptions, $G \ne C$. Hence, $|U_i| = 2$ for at least one $i \in [5]$, and so $\alpha(G) = 3$. Renaming vertices of the cycle $C$ if necessary, we may assume that $|U_1| = 2$. Since both $U_1 \cup U_3$ and $U_1 \cup U_4$ are independent sets, we infer that $|U_3| = |U_4| = 1$, that is, $U_3 = \{u_3\}$ and $U_4 = \{u_4\}$. Moreover since $U_2 \cup U_5$ is an independent set and $\alpha(G) = 3$, we note that $|U_2| = 1$ or $|U_5| = 1$ (or both $|U_2| = 1$ and $|U_5| = 1$). By symmetry, we may assume that $|U_5| = 1$, and so $U_5 = \{u_5\}$. We now let $S = \{u_1,u_2,u_4\}$. The set $S$ is a dominating set of $G$. The vertex in $U_1\setminus\{u_1\}$ is $S$-defended by $u_2$. If $|U_2| = 2$, then the vertex in $U_2\setminus\{u_2\}$ is $S$-defended by $u_1$. The vertices $u_3$ and $u_5$ are $S$-defended by $u_4$. Hence every vertex in $V(G)\setminus S$ is $S$-defended, implying that $S$ is a secure dominating set of $G$, and so $\gamma_s(G) \le |S| = 3 = \alpha(G)$.
\end{proof}	

Olariu~\cite{olariu88} proved the following structural result for paw-free graphs.

\begin{theorem}[\cite{olariu88}]
\label{paw}
A graph $G$ is \paw-free graph if and only if every component of $G$ is $C_3$-free or a complete multipartite graph.
\end{theorem}

Cockayne et al.~\cite{Cockayne05} proved that for a complete multipartite graph $G$, we have $\gamma_s(G) \le \alpha(G)$. We are now in a position to prove Theorem~\ref{thm:main4}. Recall its statement.

\medskip
\noindent \textbf{Theorem~\ref{thm:main4}}. \emph{If $G$ is a connected $(P_5,\paw)$-free graph, then $\gamma_s(G) \le \max\{3,\alpha(G)\}$.
}

\begin{proof}
Let $G$ be a connected $(P_5,\paw)$-free graph. In particular, since $G$ is paw-free, by Theorem~\ref{paw} the graph $G$ is $C_3$-free or a complete multipartite graph. If $G$ is a complete multipartite graph, then as shown in~\cite{Cockayne05} we have $\gamma_s(G) \le \alpha(G)$. If $G$ is $C_3$-free, then $G$ is a connected $(P_5, C_3)$-free graph, and so by Theorem~\ref{p5c3thm} we have $\gamma_s(G) \le \max\{3,\alpha(G)\}$.
\end{proof}

Recall that by Theorem~\ref{$C_3$-free}, if $G$ is a $C_3$-free graph, then $\gamma_s(G) \le \frac{3}{2}\alpha(G)$. Hence as an immediate consequence of this result, Theorem~\ref{paw}, and the result of Cockayne et al.~\cite{Cockayne05}, we have the following upper bound on the secure domination number of paw-free graphs.

\begin{coro}\label{Paw}
If $G$ is a paw-free graph, then $\gamma_s(G)\le \frac{3}{2}\alpha(G)$.
\end{coro}

\section{$(P_5,C_4)$-free graphs}
\label{Sect:P5C4-free}

In this section, we improve the upper bound given in  Theorem~\ref{thm:main1} on $(P_5, C_4)$-free graphs, a subclass of the class of $P_5$-free graphs, by using structural properties of such graphs. We shall prove Theorem~\ref{thm:main5} which states that if $G$ is a connected $(P_5,C_4)$-free graph, then $\gamma_s(G) \le \max\{3,\alpha(G)\}$.

We use the structural result for $(P_5,C_4)$-free graphs given by Fouquet~\cite{Fouquet}. A graph $G$ is called a \emph{complete buoy} if $V(G)$ can be partitioned as  $(A_1,A_2,A_3,A_4,A_5)$ where $A_i$ is a non-empty clique, $[A_i, A_{i+1}]$ is complete, and $[A_i, A_{i+2}]=\emptyset$ for all $i \in [5]$, where the indices follow the modulo~$5$ arithmetic. Thus such a graph $G$ is a blow-up of a $5$-cycle in the sense that we replace each vertex $a_i$ on a $5$-cycle $a_1a_2a_3a_4a_5a_1$ with a clique $A_i$ and where the closed neighborhood of every vertex $v \in A_i$ in $G$ is given by $N[v] = A_{i-1} \cup A_i \cup A_{i+1}$ for all $i \in [5]$.

A subset $X\subseteq V(G)$ with $|X|\ge 2$ is called a \emph{homogeneous set} of $G$ if for any vertex $v\in V(G)\setminus X$, either $[\{v\},X]$ is complete or $[\{v\},X]=\emptyset$.

\begin{theorem}[\cite{Fouquet}]\label{p5c4}
If $G$ is a connected $(P_5,C_4)$-free graph, then $V(G)$ can be partitioned as $V_1\cup V_2$ (one or both are possibly empty) such that the following properties are satisfied.
\begin{enumerate}
\item $G[V_1]$ is $C_t$-free for every $t\ge 4$.
\item If $V_2\ne  \emptyset$, then it can be partitioned into sets inducing maximal complete buoys and each of these sets is a homogeneous set of $G$ whose open neighborhood is a clique in $V_1$. Moreover, there is a clique in $V_1$ whose open neighborhood contains $V_2$.
\end{enumerate}
\end{theorem}

We proceed further with the following lemma.

\begin{lemma}\label{buoy}
If $G$ is a connected $(P_5,C_4)$-free graph with $\alpha(G)=2$ and $\gamma_s(G)=3$, then $G$ is isomorphic to a complete buoy.
\end{lemma}
\begin{proof}
By Theorem~\ref{p5c4}, $V(G)$ can be partitioned as $V_1\cup V_2$ with the properties as defined in Theorem~\ref{p5c4}. If $G$ is $C_5$-free, then by Theorem~\ref{cycle}, $\gamma_s(G)\le \alpha(G)$, a contradiction. Hence, $G$ contains an induced $C_5$, implying that $V_2 \ne  \emptyset$. If $B_1$ and $B_2$ induces two distinct maximal complete buoys, then by Theorem~\ref{p5c4}(b), $B_1\cup B_2\subseteq V_2$ and $[B_1,B_2]=\emptyset$. Hence if $G$ contains more than one maximal complete buoy, then $\alpha(G)\ge 4$, a contradiction. Hence, $G$ contains at most one maximal complete buoy. In fact, $G[V_2]$ is a complete buoy. Moreover, since $\alpha(G)=2$, we have $V_1\setminus N(V_2)=\emptyset$. Therefore, $V(G)=V_2\cup N(V_2)$, where $V_2$ induces a complete buoy in $G$. By Theorem~\ref{p5c4}(b), $N(V_2)$ is a clique (possibly empty) and $[V_2,N(V_2)]$ is complete. If $N(V_2)\ne  \emptyset$, then for a vertex $u\in V_2$ and a vertex $v\in N(V_2)$, the set $S=\{u,v\}$ is a secure dominating set of $G$, a contradiction to the fact that $\gamma_s(G)=3$. Hence, $N(V_2)=\emptyset$ and hence $V(G)=V_2$, that is $G$ is isomorphic to a complete buoy.
\end{proof}

We are now in a position to prove Theorem~\ref{thm:main5}. Recall its statement.

\medskip
\noindent \textbf{Theorem~\ref{thm:main5}}. \emph{If $G$ is a connected $(P_5,C_4)$-free graph, then $\gamma_s(G) \le \max\{3,\alpha(G)\}$.
}

\begin{proof}
We use induction on the number of vertices of $G$ to prove the theorem. If $|V(G)|=1$, then the bound is trivial. Assume that the result is true for every connected $(P_5,C_4)$-free graph with less than $n$ vertices. Let $G$ be a connected $(P_5,C_4)$-free graph with $n$ vertices. By Theorem~\ref{p5c4}, the vertex set $V(G)$ can be partitioned as $V_1\cup V_2$ with the properties we have $\gamma_s(G)\le \alpha(G)$. Hence we may assume that $V_2 \ne  \emptyset$. Let $B$ induce a maximal complete buoy in $G[V_2]$. If $N(B)=\emptyset$, then, since $G$ is connected, $V(G)=B$; that is, $G$ is isomorphic to a complete buoy. Thus, $\alpha(G)=2$ and $\gamma_s(G)=3$. Hence we may further assume that $N(B)\ne  \emptyset$. By Theorem~\ref{p5c4}(b), $N(B)$ is a clique in $V_1$  and hence $V_1\ne \emptyset$. By the definition of a complete buoy, we can partition $B$ as $(A_1,A_2,A_3,A_4,A_5)$ such that $A_i$ is a non-empty clique, $[A_i, A_{i+1}]$ is complete, and $[A_i, A_{i+2}]=\emptyset$ for every $i\in[5]$, where indices follow modulo~$5$ arithmetic. We note that $G[V(G)\setminus A_5]$ is a connected graph since any vertex in $N(B)$ is adjacent to every vertex in $N[A_5]$. Furthermore, $A_5$ is a clique.

Let $H_1$ and $H_2$ be the graphs $G[A_5]$ and $G[V(G)\setminus A_5]$, respectively. We note that $V(G)$ is partitioned as $V(H_1)\cup V(H_2)$. By the induction hypothesis, $\gamma_s(H_2)\le \max\{3,\alpha(H_2)\}$. Suppose, to the contrary, that $\gamma_s(H_2) > \alpha(H_2)$. In this case, we infer that $\gamma_s(H_2)=3$ and $\alpha(H_2)=2$. By Lemma~\ref{buoy}, $H_2$ is isomorphic to a complete buoy. Hence by Theorem~\ref{p5c4}(b), $V(H_2)\subseteq V_2$. Recall that $V(H_1)=A_5\subset B\subseteq V_2$, and so $V(H_1)\cup V(H_2)\subseteq V_2$. Consequently, $V(G)\subseteq V_2$, which is a contradiction to the fact that $V_1\ne  \emptyset$. Hence, $\gamma_s(H_2)\le \alpha(H_2)$.

Since any independent set of $H_2$ is also an independent set of $G$, we have $\alpha(H_2)\le \alpha(G)$. Therefore, $\gamma_s(H_2)\le \alpha(H_2)\le \alpha(G)$. To prove the theorem, it is sufficient to show that there exists a secure dominating set of $G$ of cardinality $\gamma_s(H_2)$. Let $S_{H_2}$ be a secure dominating set of $H_2$ of cardinality $\gamma_s(H_2)$. We note that $S_{H_2}\cap A_5=\emptyset$. Recall that $B\subseteq V_2$ induces a maximal complete buoy, $N(B)$ is a clique in $V_1$, and $[B,N(B)]$ is complete. We now consider the following cases. In each case, we show that there exists a secure dominating set of $G$ of cardinality $\gamma_s(H_2)$.
	
\noindent\textbf{Case 1:} $S_{H_2}\cap N(B)=\emptyset$.

If $|S_{H_2}\cap B|\le 1$, then there exists a vertex $a\in A_i$ for some $i\in[4]$ such that $N[a]\cap S_{H_2}=\emptyset$. Hence $S_{H_2}$ is not a dominating set of $H_2$, a contradiction. Hence, $|S_{H_2}\cap B|\ge 2$. Let $u,v\in S_{H_2}\cap B$, $a_2\in A_2$, and $x\in N(B)$. Such a vertex $x$ exists since $N(B)\ne  \emptyset$. We consider the set $S=(S_{H_2}\setminus \{u,v\})\cup \{a_2,x\}$, and we prove that $S$ is a secure dominating set of $G$. Since $S_{H_2}\cap N(B)=\emptyset$ and $[B,V(H_2)\setminus N[B]]=\emptyset$, any vertex in $V(H_2)\setminus (S_{H_2}\cup N[B])$  is $S_{H_2}$-defended by some vertex in $S_{H_2}\setminus\{u,v\}$. Therefore every vertex in $V(H_2)\setminus (S\cup N[B])$ is $S$-defended by some vertex in $S\setminus\{a_2,x\}$. Every vertex in $A_1,A_2\setminus\{a_2\}$, and $A_3$ is $S$-defended by $a_2$. Moreover, every vertex in $A_4$, $A_5$, and $N(B)\setminus \{x\}$ is $S$-defended by $x$. Hence every vertex in $V(G)\setminus S$ is $S$-defended. Hence, $S$ is a secure dominating set of $G$ of cardinality $\gamma_s(H_2)$.

\noindent\textbf{Case 2:} $S_{H_2}\cap N(B) \ne  \emptyset$ and $|S_{H_2}\cap B|\ge 2$.	

Let $u', v'\in S_{H_2}\cap B$ and $a'_2\in A_2$ and $a'_4\in A_4$ . We consider the set $S'=(S_{H_2}\setminus \{u', v'\})\cup \{a'_2, a'_4\}$, and prove that $S'$ is a secure dominating set of $G$. Since $S_{H_2}\cap N(B)\ne  \emptyset$ and $[B,V(H_2)\setminus N[B]]=\emptyset$, every vertex in $V(H_2)\setminus (S_{H_2}\cup B)$  is $S_{H_2}$-defended by some vertex in $S_{H_2}\setminus\{u',v'\}$. Hence every  vertex in $V(H_2)\setminus (S'\cup B)$  is $S'$-defended by some vertex in $S'\setminus\{a'_2, a'_4\}$. Every vertex in $A_1,A_2\setminus\{a'_2\}$, and $A_3$ is $S'$-defended by $a'_2$. Moreover, every vertex in $A_4\setminus\{a'_4\}$ and $A_5$ is $S'$-defended by $a'_4$. Hence every vertex in $V(G)\setminus S'$ is $S'$-defended. Therefore, $S'$ is a secure dominating set of $G$ of cardinality $\gamma_s(H_2)$.

\noindent\textbf{Case 3:} $S_{H_2}\cap N(B) \ne \emptyset$ and $|S_{H_2}\cap B|\le 1$. 	

We show firstly that $|S_{H_2}\cap N(B)|\ge 2$. Suppose firstly that $|S_{H_2}\cap N(B)|=1$. Let $x'\in S_{H_2}\cap N(B)$. Suppose, to the contrary, that $S_{H_2}\cap B=\emptyset$. We note that the vertices in $B\setminus A_5$ are dominated by $x'$ only. We also note that for any vertex $a\in B\setminus A_5$, the set $(S_{H_2}\setminus \{a\})\cup\{x'\}$ is not a dominating set of $G$. This is a contradiction to the fact that $S_{H_2}$ is a secure dominating set of $H_2$. Hence, $|S_{H_2}\cap B|=1$. Let $a'\in S_{H_2}\cap B$. We note that there exists at least one index $i\in [4]$ such that the vertices in $A_i$ are $S_{H_2}$-defended by $x'$ only. If $A_5\cap N(a') = \emptyset$, then every vertex in $A_5$ is $S_{H_2}$-defended by $x'$. If $A_5\cap N(a')\ne  \emptyset$, then every vertex in $A_5$ is $S_{H_2}$-defended by $a'$. Hence $S_{H_2}$ is a secure dominating set of $G$ of cardinality $\gamma_s(H_2)$.
	
Suppose secondly that $|S_{H_2}\cap N(B)|\ge 2$. Since $|S_{H_2}\cap B|\le 1$, there exists a vertex in $B\setminus A_5$ that is $S_{H_2}$-defended by a vertex $x\in S_{H_2}\cap N(B)$. Let $y\in S_{H_2}\cap N(B)$ be a vertex other than $x$. We note that $y$ dominates every vertex in $N[B]$. Hence every vertex in $A_5$ is $S_{H_2}$-defended by $x$. Therefore, $S_{H_2}$ is a secure dominating set of $G$ of cardinality $\gamma_s(H_2)$.
\end{proof}

\section{Conclusion}

In this paper, we have shown that if $G$ is a $P_5$-free graph, then $\gamma_s(G)\le \frac{3}{2}\alpha(G)$. We remark that the given bound is optimal since for a graph $G$ that is a disjoint union of $C_5$'s. In light of this, we point out two natural questions that need to be considered.

\begin{quest}
Can we improve the bound $\gamma_s(G)\le \frac{3}{2}\alpha(G)$ if $G$ is a connected $P_5$-free graph with $\alpha(G)\ge 3?$
\end{quest}

\begin{quest}
Can we have the bound $\gamma_s(G)\le \frac{3}{2}\alpha(G)$ if $G$ is a  $P_t$-free graph, where $t\ge 6 ? $
\end{quest}

We have also shown that if $G$ is a $(P_3 \cup P_2)$-free graph, then $\gamma_s(G)\le \alpha(G)+1$. Recall that $\alpha(C_5)=2$ and $\gamma_s(C_5)=3$. Moreover, $C_5$ is a $(P_3 \cup P_2)$-free graph. The given bound can be improved for $\alpha(G)\ge 3$. We have proved that if $G$ is a $H$-free graph, where $H \in \{P_3\cup P_1, (K_2 \cup 2K_1)\}$, then $\gamma_s(G)\le \max\{3,\alpha(G)\}$. We have also proved that if $G$ is a connected $(P_5,H)$-free graph, where $H \in \{\paw, C_4\}$, then $\gamma_s(G)\le \max\{3,\alpha(G)\}$. This bound is optimal for $\alpha(G)\ge 2$. When $\alpha(G)=2$, $C_5$ is a tight example. When $\alpha(G)\ge 3$, the star graphs are tight examples.

\end{document}